\documentclass[11pt]{amsart}
\newtheorem{theorem}{Theorem}[section]

\newtheorem{lemma}[theorem]{Lemma}
\newtheorem{proposition}[theorem]{Proposition}

\theoremstyle{definition}

\numberwithin{equation}{section}

\begin{document}

\title[Lack of interior $L^q$ bounds for stable solutions]{Lack of interior $L^q$ bounds for stable solutions to elliptic equations}
\author{Salvador Villegas}
\address{Departamento de An\'{a}lisis
Matem\'{a}tico, Universidad de Granada, 18071 Granada, Spain.}
\email{svillega@ugr.es}

\begin{abstract}
We consider stable solutions of semilinear elliptic equations of the form $-\Delta u=f(u)$ in a bounded domain $\Omega\subset\mathbb{R}^N$. In a well-known paper \cite{cfrs}, Cabré, Figalli, Ros-Oton and Serra obtained interior estimates for the $W^{1,2}$-norm of $u$ in terms of the $L^1$-norm of $u$ and proved interior Hölder regularity for dimensions $N\leq 9$. All these results rely on the assumption that $f$ is nonnegative. We show that, for general nonlinearities $f\in C^\infty(\mathbb{R})$, it is impossible, in any dimension $N\geq 1$, to obtain an interior $L^q$ estimate in terms of the $L^p$-norm of $u$ whenever $1\leq p<q\leq \infty$.
\end{abstract}

\maketitle

\section{Introduction and results}

We consider the semilinear elliptic equation

\begin{equation}\label{mainequation}
-\Delta u=f(u)\ \ \ \ \ \ \ \mbox{ in } \Omega,
\end{equation}

\noindent where $\Omega\subset\mathbb{R}^N$  is a bounded domain, $f\in C^1(\mathbb{R})$ and $u\in C(\overline{\Omega})$. By standard regularity arguments, we obtain that any solution $u \in C(\overline{\Omega})$ (no matter how weak the notion of solution is) is a classical solution, and $u \in C^2(\Omega) \cap C^1(\overline{\Omega})$.

This problem is the Euler–Lagrange equation associated with the energy functional

\begin{equation}\label{energy}
E(u):=\int_{\Omega} \left(\frac{1}{2}|\nabla u|^2-F(u)\right)dx,
\end{equation}
where $F(t)=\int_0^t f(s)ds$. The second variation of \eqref{energy} at $u$ is given by  
\begin{equation*}\label{stability}
Q_u(\xi):=\left.\frac{d^2}{d\varepsilon^2}\right|_{\varepsilon=0}E(u+\varepsilon\xi)=\int_{\Omega} \left(|\nabla \xi|^2-f'(u)\xi^2\right)dx.
\end{equation*}

We say that a solution $u$ of \eqref{mainequation} is stable if $Q_u(\xi)\geq 0$ for all $\xi\in C^1_0(\overline{\Omega})$ with $\xi=0$ on $\partial \Omega$. Stability is equivalent to the nonnegativeness of the first Dirichlet eigenvalue for $-\Delta-f'(u)$, the linearized operator of \eqref{mainequation}.

Examples of stable solutions are local minimizers of $E$. Recall that $u\in C^1(\overline{\Omega})$ is a local minimizer of $E$ if it minimizes $E$ under every sufficiently small $C^1(\overline{\Omega})$  perturbation with the same boundary values. In many cases, solutions of $-\Delta u=f(u)$ in a suitable weak sense can be obtained as pointwise limits of classical solutions. For this reason, in order to establish the regularity of such weak solutions, it is very important to obtain a priori $L^\infty$  bounds for classical solutions. Some widely studied examples in the literature are $f(u)=\lambda e^u$ and $f(u)=\lambda (1+u)^p$ for certain values of $\lambda>0$ and $p>1$ (see the works of Gelfand \cite{g} and Crandall and Rabinowitz \cite{cr}). We refer the reader to the monograph \cite{dupaigne} by Dupaigne for further results on stable solutions.

In 2020, Cabré, Figalli, Ros-Oton and Serra \cite{cfrs} proved Hölder regularity for stable solutions to \eqref{mainequation} in the optimal dimensional range $N\leq 9$. Among other results, they obtain the following result on interior regularity:

\begin{theorem}[{\cite{cfrs}, Theorem 1.2}.]\label{interior}
Let $u\in C^\infty(\overline{B_1})$ be a stable solution of $-\Delta u=f(u)$ in $B_1\subset \mathbb{R}^N$, for some nonnegative function $f\in C^1(\mathbb{R})$ .

Then,
$$\Vert\nabla u\Vert_{L^{2+\gamma}(B_{1/2})}\leq C\Vert u\Vert_{L^1(B_1)} \ \  \ \mbox{ for every } N\geq 1$$
\noindent for some dimensional constants $\gamma>0$ and $C$. In addition,
$$\Vert u\Vert_{C^\alpha (\overline{B}_{1/2})}\leq C\Vert u\Vert_{L^1(B_1)} \ \  \ \mbox{ if } N\leq 9,$$
\noindent where $\alpha>0$ and $C$ are dimensional constants.
\end{theorem}

In \cite[Open problem 1.6.1]{cabre1}, the author asks whether the nonnegativity of $f$ is necessary to obtain interior $L^\infty$ estimates. In fact, for general nonlinearities $f\in C^\infty(\mathbb{R})$ (without any additional assumptions), interior $L^\infty$ estimates can be obtained in dimensions $N\leq 4$ \cite{cabre2, cabre3} as well as in dimension $N=5$ \cite{pzz}. In fact, in dimensions $2\leq N\leq 9$, for nonlinearities that are bounded from below and exhibit at least linear growth at infinity, one also obtains interior $C^{0,\alpha}$ regularity (see \cite{fp}). Such estimates are also available in the radial case for dimensions $N\leq 9$ \cite{cc}, leaving the question open for the nonradial case in dimensions $6\leq N\leq 9$ without the hypothesis $f\geq 0$.

In any case, in all the situations mentioned above, the interior $L^\infty$ estimates are obtained in terms of the $W^{1,2}$-norm of $u$.  In this paper, we show that, for general nonlinearities $f\in C^\infty(\mathbb{R})$ (we emphasize that no further assumptions are imposed), it is impossible to obtain, in any dimension $N\geq 1$, an interior $L^q$ estimate in terms of the $L^p$-norm of $u$ whenever $1\leq p<q\leq \infty$.

\begin{theorem}\label{principal}

Let $N\geq 1$ and $1\leq p<q\leq\infty$. Then there exists a sequence $\{u_n\}_{n\in \mathbb{N}}\subset C^\infty(\overline{B_1})$ of stable solutions to problems \eqref{mainequation} with nonlinearities $\lbrace f_n\rbrace_{n\in \mathbb{N}}\subset C^\infty(\mathbb{R})$ such that:
\begin{equation*}
\frac{\Vert u_n\Vert_{L^q (B_{1/2})}}{\Vert u_n\Vert_{L^p(B_1)}}\longrightarrow +\infty\ \ \ \ \ \text{as}\ \ n\rightarrow\infty
\end{equation*} 
\end{theorem}

\begin{theorem}\label{gradiente}

Let $N\geq 1$, $1\leq p\leq\infty$ and $1\leq q \leq \infty$ with $(p,q)\neq (\infty,1)$. Then there exists a sequence $\{u_n\}_{n\in \mathbb{N}}\subset C^\infty(\overline{B_1})$ of stable solutions to problems \eqref{mainequation} with nonlinearities $\lbrace f_n\rbrace_{n\in \mathbb{N}}\subset C^\infty(\mathbb{R})$ such that:
\begin{equation*}
\frac{\Vert\nabla u_n\Vert_{L^q (B_{1/2})}}{\Vert u_n\Vert_{L^p(B_1)}}\longrightarrow +\infty\ \ \ \ \ \text{as}\ \ n\rightarrow\infty
\end{equation*} 
\end{theorem}

It is easy to see that if a stable function of one variable is regarded as a function of several variables, stability is preserved in higher dimensions. For this reason it suffices to prove our results in dimension $N=1$. The following proposition provides a large class of examples of stable one-dimensional functions.

\begin{proposition}\label{muygeneral}
Let $\phi \in C^\infty((0,1])$ such that $\phi'<0$. Then, for every $x_0\in (0,1)$ there exist $u\in C^\infty([-1,1])$ and $f\in C^\infty (\mathbb{R})$ such that $u$ is an even stable solution of $-u''=f(u)$ in $(-1,1)$ and

\begin{equation}\label{muyfuerte}
\begin{array}{ll}
u(x)<\phi(x) & \text{   if    } \ 0\leq x<x_0,\\
u(x)=\phi(x)  & \text{   if    } \ x_0\leq x\leq 1.      
\end{array}
\end{equation}

\end{proposition}

The main idea in the proof of Theorem \ref{principal} (the same argument applies to Theorem  \ref{gradiente}) is the following.  If $1\leq p<q\leq \infty$, we choose $\phi$ in Proposition \ref{muygeneral} such that $\phi\in L^p\setminus L^q$. Then, by taking $x_0$ sufficiently small, we obtain that the $L^p$-norm of $u$ remains uniformly bounded on $(0,1)$, while the $L^q$-norm of $u$ on $(x_0,1)$ is arbitrarily large.

\section{Proof of the lack of interior $L^q$ bounds}

\begin{lemma}\label{estables1}

Let $u \in C^\infty([-1,1])$ be an even function such that for some $K>0$ and $\varepsilon_0 \in (0,1)$,
\begin{equation}\label{property}
\begin{array}{ll}
-u'(x)=Kx  & \text{for } 0<x\leq \varepsilon_0,\\
-u'(x)\geq Kx  & \text{for } \varepsilon_0<x<1.
\end{array}
\end{equation}

Then there exists $f \in C^\infty(\mathbb{R})$ such that $u$ is a stable solution of 
$-u''=f(u)$ in $(-1,1)$.

\end{lemma}

\begin{proof}
We first prove that there exists $f \in C^\infty(\mathbb{R})$ such that 
$u$ solves $-u''=f(u)$ in $[-1,1]$. 
Since $u'<0$ in $(0,1]$, we define
\[
f(s):=-u''(u^{-1}(s)), \qquad s\in [u(1),u(0)),
\]
which yields $f\in C^\infty([u(1),u(0)))$.
Since $f(s)=K$ for $s\in [u(\varepsilon_0),u(0))$, we extend $f$ to $\mathbb{R}$
by setting $f(s)=K$ for $s\ge u(0)$ and extending it in a $C^\infty$ way
to $(-\infty,u(1))$.
Finally, since $u$ is even, it follows that $-u''=f(u)$ in $[-1,1]$.

To prove the stability of $u$, let $\xi\in C^1([-1,1])$ with 
$\xi(-1)=\xi(1)=0$. 
Since $f'(u(x))=0$ for $0<x<\varepsilon_0$ and 
$f'(u(x))=-u'''(x)/u'(x)$ for $\varepsilon_0<x<1$, 
integrating by parts yields

\begin{equation}\label{ss}
\begin{aligned}
\int_0^1 (\xi'^2 - f'(u)\xi^2) dx 
&\geq \int_{\varepsilon_0}^1 \left(\xi'^2 +\frac{u'''}{u'}\xi^2\right) dx \\
&= \int_{\varepsilon_0}^1 \xi'^2 dx
+\left[ u''(x)\frac{\xi (x)^2}{u'(x)}\right]_{\varepsilon_0}^{1}
-\int_{\varepsilon_0}^1 u''\left( \frac{\xi^2}{u'}\right)' dx \\
&= \int_{\varepsilon_0}^1 \xi'^2 dx
- (-K)\frac{\xi (\varepsilon_0)^2}{-K \varepsilon_0}
-\int_{\varepsilon_0}^1 u''\left( \frac{2\xi\xi 'u'-\xi^2 u''}{u'^2}\right) dx \\
&= \int_{\varepsilon_0}^1\left(\xi'-\frac{u''}{u'}\xi\right)^2 dx-\frac{\xi(\varepsilon_0)^2}{\varepsilon_0}.
\end{aligned}
\end{equation}

On the other hand, using \eqref{property} and the Cauchy-Schwarz inequality, we have

\begin{equation*}
\begin{aligned}
\frac{\xi(\varepsilon_0)^2}{K^2 \varepsilon_0^2}
&=\left(\int_{\varepsilon_0}^1\left(\frac{\xi}{u'}\right)' dx\right)^2 =\left(\int_{\varepsilon_0}^1\frac{1}{u'}\left(\xi'-\frac{u''}{u'}\xi\right) dx\right)^2\\
&\leq \int_{\varepsilon_0}^1\frac{1}{u'^2}dx \int_{\varepsilon_0}^1\left(\xi'-\frac{u''}{u'}\xi\right)^2dx\leq \int_{\varepsilon_0}^1\frac{1}{K^2 x^2}dx \int_{\varepsilon_0}^1\left(\xi'-\frac{u''}{u'}\xi\right)^2dx\\
&=\frac{1}{K^2}\left( \frac{1}{\varepsilon_0}-1\right) \int_{\varepsilon_0}^1\left(\xi'-\frac{u''}{u'}\xi\right)^2dx\leq \frac{1}{K^2 \varepsilon_0}\int_{\varepsilon_0}^1\left(\xi'-\frac{u''}{u'}\xi\right)^2dx .
\end{aligned}
\end{equation*}

Therefore
\[
\frac{\xi(\varepsilon_0)^2}{\varepsilon_0}\le
\int_{\varepsilon_0}^1\left(\xi'-\frac{u''}{u'}\xi\right)^2 dx .
\]
Hence, by \eqref{ss}, we deduce that
\[
\int_0^1 (\xi'^2 - f'(u)\xi^2) dx \ge 0 .
\]

By symmetry we also have $\int_{-1}^0 (\xi'^2 - f'(u)\xi^2) dx \ge 0$ and therefore $u$ is stable.

\end{proof}

\

{\em Proof of Proposition \ref{muygeneral}}. Since $\phi'<0$ in $[\frac{x_0}{2},1]$ we can take $K>0$ such that

\begin{equation}\label{cotaK}
-\phi'(x)\geq 3Kx \ \ \text{for } \frac{x_0}{2}\leq x\leq 1.
\end{equation}

Define an even function $u\in C^\infty([-1,1])$ by

\begin{equation}\label{defu}
\begin{aligned}
-u'(x)=&\left( 1-\Psi(x)\right) Kx+\Psi(x)(-\phi'(x))   & \text{              for } 0<x\leq 1,\\ 
 u(1)=&\, \phi (1),&
\end{aligned}
\end{equation}

\noindent where $\Psi\in C^\infty (\mathbb{R})$ is a nondecreasing function satisfying $\Psi(x)=0$ if $x\leq x_0/2$, $\Psi(x)=1$ if $x\geq x_0$, $0<\Psi(x)<1$ if $x_0/2<x<x_0$ and $\Psi(3x_0/4)=1/2$. For instance, take

$$\Psi(x):=\frac{\int_{x_0/2}^{x} e^{-\left( \frac{1}{(t-x_0)^2}+ \frac{1}{(t-x_0/2)^2}\right)}dt}{\int_{x_0/2}^{x_0} e^{-\left( \frac{1}{(t-x_0)^2}+ \frac{1}{(t-x_0/2)^2}\right)}dt}\ \ \ \text{ for } x_0/2<x<x_0,$$

\noindent and extend it to $\mathbb{R}$ by $0$ in $(-\infty,x_0/2]$ and by $1$ in $[x_0,+\infty)$. 

To prove the stability of $u$ we will apply Lemma  \ref{estables1} with $\varepsilon_0=x_0/2$. It is obvious that $-u'(x)=Kx$ for $0<x\leq x_0/2$. On the other hand, if $x\in[x_0,1]$, then $-u'(x)=-\phi'(x)\geq 3Kx> Kx$. Finally, if $x\in (x_0/2,x_0)$, taking into account that $Kx<-\phi'(x)$ and $0<\Psi(x)<1$ we conclude  $Kx< -u'(x)<-\phi'(x)$ and this proves that $u$ is a stable solution of $-u''=f(u)$ in $[-1,1]$ for some $f \in C^\infty(\mathbb{R})$.

We now prove \eqref{muyfuerte}. Clearly $u(x)=\phi(x)$ for $x_0\leq x\leq 1$. In addition, as mentioned before,  $-u'(x)<-\phi'(x)$ for $x_0/2<x<x_0$ and it follows easily that  $u(x)<\phi(x)$ for $x_0/2< x<x_0$.

It remains to show that $u(x)<\phi(x)$ for $0\leq x\leq x_0/2$. To this end, we first observe that $0\leq \Psi(x)\leq 1/2$ and $Kx<-\phi'(x)$ for $x\in [x_0/2,3x_0/4]$. Therefore

$$-u'(x)\leq \frac{Kx-\phi'(x)}{2} \ \ \ \ \ \ \ \ \mbox{    for   } \  x\in \left[\frac{x_0}{2},\frac {3x_0}{4}\right].$$

Applying this and \eqref{cotaK} yields

$$u'(x)-\phi'(x)\geq \frac{-Kx-\phi'(x)}{2}\geq Kx \ \ \ \ \ \ \ \ \mbox{    for   } \  x\in \left[\frac{x_0}{2},\frac {3x_0}{4}\right].$$

From this, taking into account that $u(x_0)=\phi(x_0)$ and $-u'(x)<-\phi'(x)$ for $x\in [3x_0/4,x_0]$; $-\phi'(x)>0$ and $-u'(x)=Kx$ for  $x\in (0,x_0/2]$, we conclude that

\[
\phi(x)-u(x)=\int_x^{x_0}(u'(t)-\phi'(t))\,dt
> \int_x^{3x_0/4}(u'(t)-\phi'(t))\,dt
\]
\[
> \int_0^{x_0/2}u'(t)\,dt+\int_{x_0/2}^{3x_0/4}(u'(t)-\phi'(t))\,dt
\]
\[
\ge \int_0^{x_0/2}(-Kt)\,dt+\int_{x_0/2}^{3x_0/4}Kt\,dt
=\frac{Kx_0^2}{32}>0
\qquad \text{for } 0\le x\le x_0/2 ,
\]

\noindent and the proof is complete. 

\qed

\

As mentioned in the introduction the idea of the proof of the main results is to consider functions that depend on a single variable. Thus, if the one dimensional solution is stable, then the corresponding function in several variables is also stable. More precisely, if $u \in C^\infty([-1,1])$ is a stable solution of $-u'' = f(u)$ in $(-1,1)$ for some $f \in C^\infty(\mathbb{R})$, then
\[
U(x_1,\ldots,x_N):=u(x_1)
\]
is a stable solution of $-\Delta U = f(U)$ in $(-1,1)^N$ for the same function $f$.

Indeed, let $\xi \in C^1([-1,1]^N)$ be such that $\xi=0$ on $\partial(-1,1)^N$. Then, for any $(x_2,\ldots,x_N)\in(-1,1)^{N-1}$, the one-dimensional function $\xi(\cdot,x_2,\ldots,x_N)$ belongs to $C^1([-1,1])$ and vanishes at $-1$ and $1$. Hence, by the stability of $u$ we conclude that 

\[
\begin{aligned}
\int_{(-1,1)^N} |\nabla \xi|^2\,dx
&\ge \int_{(-1,1)^N} \left(\frac{\partial \xi}{\partial x_1}\right)^2 dx \\
&= \int_{(-1,1)^{N-1}}
\left(
\int_{-1}^{1}
\left(\frac{\partial \xi}{\partial x_1}\right)^2 dx_1
\right)
dx_2 \cdots dx_N \\
&\ge \int_{(-1,1)^{N-1}}
\left(
\int_{-1}^{1} f'(u(x_1))\,\xi^2 \, dx_1
\right)
dx_2 \cdots dx_N \\
&= \int_{(-1,1)^N} f'(U)\,\xi^2\,dx ,
\end{aligned}
\]

\noindent  and the stability of $U$ is proved in $(-1,1)^N$, and thus also in $B_1 \subset (-1,1)^N$.

\

{\em Proof of Theorem \ref{principal}}. Suppose by contradiction that there exists a constant  $C=C(N,p,q)$ depending on $N$, $p$ and $q$ such that 

\begin{equation}\label{contradiction}
\Vert U\Vert_{L^q (B_{1/2})}\leq C \Vert U\Vert_{L^p (B_1)},
\end{equation}

\noindent for every stable solution $U\in C^\infty(\overline{B_1})$ to a problem of the type \eqref{mainequation}.

We first consider the case $1\leq p<q=\infty$. To do this, consider an arbitrary $x_0\in (0,1/2)$ and take $\phi(x)=\vert \log x\vert$ in Proposition \ref{muygeneral}. Define $U(x_1,x_2,...,x_N):=u(x_1)$ in $B_1$, where $u$ is the function obtained in this proposition. Hence

\begin{equation}\label{yaesta}
\Vert U\Vert_{L^p(B_1)}^p\leq \Vert U\Vert_{L^p((-1,1)^N)}^p=2^N\int_0^1 \vert u(x_1)\vert^p dx_1\leq 2^N\int_0^1\vert\log x_1\vert^p dx_1,
\end{equation}

\noindent where the last expression is a constant depending on $N$ and $p$, but not on $x_0\in (0,1/2)$.

On the other hand

$$\Vert U\Vert_{L^\infty(B_{1/2})}\geq \vert U(x_0,0,...,0)\vert=\vert \log x_0\vert \longrightarrow +\infty\ \ \ \ \ \text{as}\ \ x_0\rightarrow 0,$$

\noindent which contradicts \eqref{contradiction}.

Consider now the case  $1\leq p<q<\infty$. Choose an arbitrary $x_0\in (0,1/(2\sqrt{N}))$ and take $\phi(x)=x^{-1/q}-1$ in Proposition \ref{muygeneral}. Similar to the previous case define $U(x_1,x_2,...,x_N):=u(x_1)$ in $B_1$, where $u$ is the function obtained in this proposition. Therefore

$$\Vert U\Vert_{L^p(B_1)}^p\leq \Vert U\Vert_{L^p((-1,1)^N)}^p=2^N\int_0^1 \vert u(x_1)\vert^p dx_1\leq 2^N\int_0^1x_1^{-p/q} dx_1=\frac{2^N q}{q-p.}$$

Again this expression only depends on $N$, $p$ and $q$, but not on $x_0\in (0,1/(2\sqrt{N}))$.

Observe that $B_{1/2}\supset (x_0,1/(2\sqrt{N}))\times (0,1/(2\sqrt{N}))^{N-1}$. It follows that

\begin{align*}
\Vert U\Vert_{L^q(B_{1/2})}
&\ge \Vert U\Vert_{L^q\left((x_0,\tfrac{1}{2\sqrt N})\times (0,\tfrac{1}{2\sqrt N})^{N-1}\right)} \\
&= \frac{1}{(2\sqrt N)^{\frac{N-1}{q}}}
\Vert x_1^{-1/q}-1\Vert_{L^q(x_0,1/(2\sqrt N))} \\
&\ge \frac{1}{(2\sqrt N)^{\frac{N-1}{q}}}
\left(
\Vert x_1^{-1/q}\Vert_{L^q(x_0,1/(2\sqrt N))}
-
\Vert 1\Vert_{L^q(x_0,1/(2\sqrt N))}
\right) \\
&= \frac{1}{(2\sqrt N)^{\frac{N-1}{q}}}
\left(
\left(\log\frac{1}{2\sqrt N}-\log x_0\right)^{1/q}
-
\left(\frac{1}{2\sqrt N}-x_0\right)^{1/q}
\right)
\to +\infty
\end{align*}

\noindent as \(x_0\to0\), obtaining again a contradiction with  \eqref{contradiction}.

\qed

\

{\em Proof of Theorem \ref{gradiente}}. The proof is similar in spirit to the previous one. Arguing by contradiction, suppose that there exists a constant  $C=C(N,p,q)$ depending on $N$, $p$ and $q$ such that 

\begin{equation}\label{contradictiongradiente}
\Vert\nabla U\Vert_{L^q (B_{1/2})}\leq C \Vert U\Vert_{L^p (B_1)},
\end{equation}

\noindent for every stable solution $U\in C^\infty(\overline{B_1})$ to a problem of the type \eqref{mainequation}.

We first prove the theorem for the case $q=1$ and $1\leq p<\infty$. As the first case of the proof of Theorem \ref{principal}, consider an arbitrary $x_0\in (0,1/(2\sqrt{N}))$ and take $\phi(x)=\vert \log x\vert$ in Proposition \ref{muygeneral}. Take now $U(x_1,x_2,...,x_N):=u(x_1)$ in $B_1$, where $u$ is the function obtained in this proposition. Hence we have exactly \eqref{yaesta}, obtaining that $ \Vert U\Vert_{L^p (B_1)}$ is bounded by a constant depending on $N$ and $p$, but not on $x_0\in (0,1/(2\sqrt{N}))$.

On the other hand, using again that  $B_{1/2}\supset (x_0,1/(2\sqrt{N}))\times (0,1/(2\sqrt{N}))^{N-1}$, we conclude that

\begin{align*}
\Vert \nabla U\Vert_{L^1(B_{1/2})}
&\ge \Vert \nabla U\Vert_{L^1\left((x_0,\tfrac{1}{2\sqrt N})\times (0,\tfrac{1}{2\sqrt N})^{N-1}\right)} \\
&= \frac{1}{(2\sqrt N)^{N-1}}
\Vert x_1^{-1}\Vert_{L^1(x_0,1/(2\sqrt N))} \\
&= \frac{1}{(2\sqrt N)^{N-1}}
\left(\log\frac{1}{2\sqrt N}-\log x_0\right)
\to +\infty\ \ \ \ \ \text{as}\ \ x_0\rightarrow 0,
\end{align*}

\noindent contradicting \eqref{contradictiongradiente}.

In the case $1<q\leq \infty$, since $L^\infty (B_1)\subset L^p(B_1)$ for $1\leq p\leq \infty$, there is no loss of generality in assuming $p=\infty$. In this case take $x_0\in (0,1/(2\sqrt{N}))$ and consider $\phi(x)=1-x^{1-1/q}$ in Proposition \ref{muygeneral}. As before, define $U(x_1,x_2,...,x_N):=u(x_1)$ in $B_1$, where $u$ is the function obtained in this proposition. Therefore

$$\Vert U\Vert_{L^\infty(B_1)}\leq \Vert 1-\vert x_1\vert^{1-1/q}\Vert_{L^\infty(B_1)}=1.$$

We also obtain

\begin{align*}
\Vert \nabla U\Vert_{L^q(B_{1/2})}
&\geq \Vert \nabla U\Vert_{L^q\left((x_0,\tfrac{1}{2\sqrt N})\times (0,\tfrac{1}{2\sqrt N})^{N-1}\right)} \\
&= \frac{1}{(2\sqrt{N})^\frac{N-1}{q}}
\Vert -(1-1/q) x_1^{-1/q}\Vert_{L^q(x_0,1/(2\sqrt N))} \\
&= \frac{1}{(2\sqrt N)^\frac{N-1}{q}}(1-1/q)
\left(\log\frac{1}{2\sqrt N}-\log x_0\right)^{1/q}
\to +\infty
\end{align*}

\noindent as $x_0\to0$, which again contradicts \eqref{contradictiongradiente}.

\qed


\begin{thebibliography}{99}

\bibitem{cabre3} X. Cabré,  {\em A new proof of the boundedness results for stable solutions to semilinear elliptic equations}, Discrete Contin. Dyn. Syst. 39 (2019), no. 12, 7249-7264.

\bibitem{cabre2} X. Cabré,  {\em Regularity of minimizers of semilinear elliptic problems up to dimension 4}, Comm. Pure Appl. Math. 63 (2010), no. 10, 1362-1380

\bibitem{cabre1} X. Cabré,  {\em Hölder regularity of stable solutions to elliptic equations up to  $\mathbb{R}^9$: full quantitative proofs}, Bull. Amer. Math. Soc. (N.S.) 63 (2026), no. 1, 1–57.

\bibitem{cc} X. Cabr\'e and  A. Capella, {\em Regularity of radial minimizers and extremal solutions of semilinear elliptic equations}, J. Funct. Anal. 238 (2006), no.2, 709-733.

\bibitem{cfrs} X. Cabré, A. Figalli, X. Ros-Oton and J. Serra, {\em Stable solutions to semilinear elliptic equations are smooth up to dimension 9}, Acta Math. 224 (2020), 187–252.

\bibitem{cr} M. G. Crandall and P.H. Rabinowitz, {\em Some continuation and variational methods for positive solutions of nonlinear elliptic eigenvalue problems}, Arch. Rational Mech. Anal. 58 (1975), 207-218.

\bibitem{dupaigne} L. Dupaigne, {\em Stable Solutions of Elliptic Partial Differential Equations}, Monographs and surveys in Pure and Applied Math (2011).

\bibitem{g} I. M. Gelfand, {\em Some problems in the theory of quasilinear equations}, Amer. Math. Soc. Transl. (2) 29 (1963), 295-381.

\bibitem{fp} F. Peng, {\em The boundedness of stable solutions to semilinear elliptic equations with linear lower bound on nonlinearities}, Int. Math. Res. Not. IMRN (2024), no. 8, 6350–6373.  

\bibitem{pzz} F. Peng, Y. R.-Y. Zhang, and Y. Zhou, {\em Interior Hölder regularity for stable solutions to semilinear elliptic equations up to dimension 5}, Preprint, arXiv:2204.06345 [math.AP], (2022).

\end{thebibliography}
\end{document}